\documentclass[11pt,a4paper]{article}

\usepackage[T1]{fontenc}
\usepackage[utf8]{inputenc}
\usepackage{lmodern}
\usepackage{amsmath,amssymb,amsthm,mathtools,cite}
\usepackage{enumitem}
\usepackage{tikz}
\usetikzlibrary{positioning,fit,calc,arrows.meta,backgrounds,shapes.geometric}
\usepackage{microtype}
\usepackage{fullpage}
\usepackage{authblk}

\newtheorem{theorem}{Theorem}[section]
\newtheorem{corollary}[theorem]{Corollary}

\newtheorem{lemma}[theorem]{Lemma}
\theoremstyle{definition}

\newtheorem{question}[theorem]{Question}
\theoremstyle{remark}

\DeclareMathOperator{\ccl}{ccl}
\newcommand{\N}{\mathbb N}
\newcommand{\eps}{\varepsilon}
\allowdisplaybreaks

\begin{document}

\title{\bf Large Complete Minors from a Cheeger Condition}

\author[1]{Chengli Li\thanks{Email: lichengli@m.scnu.edu.cn.}}
\author[2]{Leyou Xu\thanks{Email: leyouxu@m.scnu.edu.cn.}}
\author[1]{Bo Zhou\thanks{Email: zhoubo@m.scnu.edu.cn.}}
\affil[1]{\footnotesize School of Mathematical Sciences, South China Normal University, Guangzhou 510631, P.R. China}
\affil[2]{\footnotesize School of Computer Science, South China Normal University, Guangzhou 510631, P.R. China}
\date{}
\maketitle

\begin{abstract}
Let $G$ be a graph on $n$ vertices, and let
$e_G(S,V(G)\setminus S)$ be the number of edges with exactly one
endpoint in $S$. The Cheeger constant and the restricted Cheeger
constant of $G$, where $k\ge1$ is real, are, respectively,
\[
 h(G)=\min_{\substack{\emptyset\ne S\subseteq V(G)\\|S|\le \frac{n}{2}}}
 \frac{e_G(S,V(G)\setminus S)}{|S|}
  \text{ and } 
 h_k(G)=\min_{\substack{\emptyset\ne S\subseteq V(G)\\
 |S|\le\min\{k,\frac{n}{2}\}}}
 \frac{e_G(S,V(G)\setminus S)}{|S|}.
\]
The contraction clique number $\ccl(G)$ is the largest integer $r$
such that $G$ contains the complete graph $K_r$ as a minor.
Krivelevich and Nenadov [Complete minors in graphs without sparse
cuts, Int. Math. Res. Not. IMRN 12 (2021) 8996--9015] proved
that, for every fixed $\eps>0$ and all sufficiently large $n$ and $d$, if $G$ is a graph on $n$ vertices with maximum degree at most $d$, then $h(G)\ge\eps d$ and
$h_{\eps n}(G)\ge(\frac{1}{2}+\eps)d$ imply
$\ccl(G)=\Omega_\eps(\sqrt{\frac{nd}{\log d}})$. They asked to determine if one can guarantee the same lower bound on $\ccl(G)$ without the additional
condition on $h_{\eps n}(G)$. They showed that this is the case when $d$ is a constant. We answer this question
affirmatively. For every $\eps>0$, there are constants
$\beta=\beta(\eps)>0$ and $n_0=n_0(\eps)$ such that, whenever
$d\ge 3$ is an integer, for every graph $G$ with
$n\ge n_0$ vertices and maximum degree at most $d$, if $h(G)\ge\eps d$, then
$\ccl(G)\ge\beta\sqrt{\frac{nd}{\log d}}$. The dependence of this lower bound on $n$ and $d$ is best possible up to a constant factor. As a corollary, a lower bound is derived for the contraction clique number of $d$-regular graphs for which the second largest eigenvalue is bounded away from $d$, compared to earlier $\frac{d}{2}$. 
The proof combines spectral properties of graphs with an analysis of lazy random walks.

\smallskip
\noindent\textbf{Keywords:}  contraction clique number, Cheeger constant, restricted Cheeger constant, branch sets, hitting sets

\noindent\textbf{2020 Mathematics Subject Classification:} 05C83, 05C50, 05C48, 05C81
\end{abstract}

\section{Introduction}

All graphs in this paper are finite, simple and undirected. All
logarithms are natural. 
A graph $H$ is a \emph{minor} of a graph $G$, written
$H\prec G$, if $H$ can be obtained from $G$ by deleting
vertices, deleting edges, and contracting edges. Contracting an edge
means identifying its endpoints, deleting every resulting loop, and
replacing parallel edges by a single edge. Equivalently,
$H\prec G$ if there is a family
$(V_x)_{x\in V(H)}$ of disjoint nonempty subsets of $V(G)$
such that $V_x$ induces a  connected subgraph for every $x\in V(H)$,  and whenever
$xy\in E(H)$, $G$ has an edge between $V_x$ and  $V_y$. These sets are the \emph{branch sets} of $G$ with respect to $H$.
The \emph{contraction clique number} 
%(also known as \emph{Hadwiger number}) 
of $G$ is defined as 
$\ccl(G)=\max\{t\in\N\mid K_t\prec G\}$.

Minors have natural importance in modern graph theory and a rich body of results and conjectures in the literature provides sufficient conditions for the existence of large complete minors.  
For example, Hadwiger \cite{Hadwiger} conjectured in 1943   that for any graph $G$, $\chi(G)\le \ccl(G)$, where $\chi(G)$ is the chromatic number of $G$. This conjecture is known when $\chi(G)\le6$: Wagner \cite{Wagner} and  Robertson, Sanders, Seymour and Thomas\cite{RSST}  connected the case $\chi(G)=5$ with the Four Colour Theorem; 
Robertson, Seymour and Thomas \cite{RST} established the case $\chi(G)=6$. For general $t$, Norin, Postle and Song \cite{NPS} obtained the first asymptotic improvement over the colouring bound supplied by the Kostochka--Thomason density theorem, proving that every $K_t$-minor-free graph is $O(t(\log t)^a)$-colourable for every $a>\frac14$. Delcourt and Postle \cite{DP} subsequently improved this to $O(t\log\log t)$ colours and reduced the linear form of Hadwiger's conjecture to the corresponding problem for suitably small graphs. 
Another example
asks which (edge) density or expansion hypotheses force a large complete minor. Mader \cite{Mader} initiated the systematic study of the extremal function for complete minors. Kostochka~\cite{Kostochka} and Thomason~\cite{Thomason1984} proved that every graph of average degree $D$ contains a complete minor of order $\Omega(D/\sqrt{\log D})$. Later, Thomason ~\cite{Thomason2001}  determined the asymptotic extremal constant and showed that random graphs provide extremal examples. Recently, Alon, Krivelevich and Sudakov ~\cite{AKS}  gave a short proof of the order of magnitude. Stronger conclusions can be obtained when sparse cuts are excluded for random graphs and random regular graphs~\cite{BCE,FKO08,FKO09}, graphs of large girth~\cite{KuOs,DR}, graphs with forbidden subgraphs~\cite{BFS}, and random subgraphs of graphs with large minimum degree~\cite{EKK}. Krivelevich and Sudakov \cite{KS} obtained large complete minors from vertex-expansion assumptions.

For a graph $G$ and disjoint sets $A,B\subseteq V(G)$, let $e_G(A,B)$ be the number of edges of $G$ with one endpoint in $A$ and the other in $B$. For every nonempty $S\subseteq V(G)$, the \emph{edge expansion} of $S$ is defined as $h_G(S)=\frac{e_G(S,V(G)\setminus S)}{|S|}$. The \emph{Cheeger constant} of $G$ is
\[
  h(G)=\min\left\{h_G(S)\,\middle|\,\emptyset\ne S\subseteq V(G),\ |S|\le\frac{|V(G)|}{2}\right\}.
\]
If $h(G)>0$, then $G$ is connected. For every real number $k\ge1$, the \emph{restricted Cheeger constant} is \cite{KN}
\[
  h_k(G)=\min\left\{h_G(S)\,\middle|\,\emptyset\ne S\subseteq V(G),\ |S|\le\min\left\{k,\frac{|V(G)|}{2}\right\}\right\}.
\]
Thus $h_k(G)=h(G)$ when $k\ge \frac{|V(G)|}{2}$.

Krivelevich and Nenadov \cite{KN} established the following two theorems.

\begin{theorem} 
[Krivelevich--Nenadov~{\cite[Theorem~1.1]{KN}}]
\label{thm:KN-strong}
For every $\eps>0$, there are constants $\beta>0$ and
$n_0,d_0\in\N$ such that the following holds for all $n\ge n_0$ and
$d\ge d_0$. Let $G$ be a graph on $n$ vertices with maximum degree
at most $d$. If $h(G)\ge\eps d$ and
$h_{\eps n}(G)\ge(\frac{1}{2}+\eps)d$, then
$\ccl(G)\ge\beta\sqrt{\frac{nd}{\log d}}$.
\end{theorem}

\begin{theorem} 
[Krivelevich--Nenadov~{\cite[Theorem~1.2]{KN}}]
\label{thm:KN-basic}
For every $\eps>0$, there are constants $\beta>0$ and $n_0\in\N$
such that the following holds. Let $G$ be a graph on $n\ge n_0$
vertices with maximum degree at most $d\ge3$. If
$h(G)\ge\eps d$, then $\ccl(G)\ge\beta\sqrt n$.
\end{theorem}

The conclusion of Theorem \ref{thm:KN-basic}  is weaker than that of Theorem \ref{thm:KN-strong}, while
Theorem \ref{thm:KN-strong} has the stronger edge expansion assumption 
$h_{\eps n}(G)\ge(\frac{1}{2}+\eps)d$. Krivelevich and Nenadov \cite{KN}  pointed out that without any doubt, Theorem \ref{thm:KN-strong} would be aesthetically more appealing if the same order of $\ccl(G)$ would hold without the stronger condition
$h_{\eps n}(G)\ge(\frac{1}{2}+\eps)d$. 
%They asked whether the second Cheeger condition in
%Theorem~\ref{thm:KN-strong} can be omitted without weakening the
%conclusion.

\begin{question}
[Krivelevich--Nenadov~{\cite[Question~4.1]{KN}}]
\label{ques:KN}
Let $G$ be a graph on $n\ge n_0$ vertices with maximum degree at most
$d=d(n)\ge3$ and $h(G)\ge\eps d$, where $\eps>0$ is a constant. Is
it true that $\ccl(G)=\Omega(\sqrt{\frac{nd}{\log d}})$?
\end{question}

The following theorem answers Question~\ref{ques:KN} affirmatively.

\begin{theorem}\label{thm:main}
For every $\eps>0$, there are constants $\beta=\beta(\eps)>0$ and $n_0=n_0(\eps)\in\N$ such that the following holds. 
 Let $G$ be a graph on $n\ge n_0$ vertices with maximum degree
at most $d$, where $d$ is an integer with $d\ge 3$. If $h(G)\ge\eps d$, then $\ccl(G)\ge \beta\sqrt{\frac{nd}{\log d}}$.
\end{theorem}

The bound in Theorem~\ref{thm:main} is best possible up to a constant
factor. Indeed, Krivelevich and Nenadov~\cite{KN}
observed that, for a sufficiently large constant \(C\) with 
\(C/n\le p<1/2\), the binomial random graph \(G(n,p)\) contains with
high probability a large induced subgraph \(G'\) satisfying even the
stronger hypotheses of Theorem~\ref{thm:KN-strong}, with maximum degree at most 
\(d'=\lceil 1.1np\rceil\) and some constant \(\eps>0\). On the other
hand, the upper bound of Fountoulakis, K\"uhn and
Osthus~\cite{FKO08} gives
\[
 \ccl(G')
 \le \ccl(G(n,p))
 =O\!\left(\sqrt{\frac{n^2p}{\log(np)}}\right).
\]
Since \(|V(G')|=\Theta(n)\) and \(d'=\Theta(np)\), this upper bound is
of order
$
 O\!\left(
   \sqrt{\frac{|V(G')|d'}{\log d'}}
 \right),
$
showing that the dependence on \(n\) and \(d\) in
Theorem~\ref{thm:main} cannot be improved beyond a constant factor.

Theorem~\ref{thm:main} also strengthens the spectral consequence of Krivelevich and Nenadov~\cite[Corollary~1.3]{KN}, which assumes the stronger inequality $\lambda_2(G)<(\frac12-\eps)d$.

\begin{corollary}\label{cor:spectral}
For every $\eps\in(0,1]$, there are constants $\beta=\beta(\eps)>0$ and $n_0=n_0(\eps)\in\N$ such that the following holds. Let $G$ be an $n$-vertex $d$-regular graph with $n\ge n_0$ and $d\ge3$, and let $\lambda_2(G)$ be the second largest eigenvalue of its adjacency matrix. If
$\lambda_2(G)\le(1-\eps)d$, then
$\ccl(G)\ge\beta\sqrt{\frac{nd}{\log d}}$.
\end{corollary}

\begin{proof}
Let $A$ be the adjacency matrix of $G$. For a set $S\subseteq V(G)$, let $\boldsymbol{1}_S\in\mathbb R^{V(G)}$ denote its indicator vector, and let $\boldsymbol{1}$ denote the all-ones vector. For a vector $\boldsymbol{y}=(y_v)_{v\in V(G)}$, we write
\[
 \|\boldsymbol{y}\|_2
 =\left(\sum_{v\in V(G)}y_v^2\right)^{1/2}
\]
for its Euclidean norm. Fix a nonempty set $S\subseteq V(G)$ with $|S|\le\frac{n}{2}$, and let
$\boldsymbol{x}=\boldsymbol{1}_S-\frac{|S|}{n}\boldsymbol{1}$. Then $\boldsymbol{x}$ is orthogonal to $\boldsymbol{1}$ and 
\[
\|\boldsymbol{x}\|_2^2=|S|(1-|S|/n)^2+(n-|S|)(|S|/n)^2=|S|(n-|S|)/n.
\]
Since $\lambda_2(G)$ is the largest eigenvalue of $A$ on the subspace orthogonal to $\boldsymbol{1}$, the Rayleigh quotient gives
\[
\begin{aligned}
 e_G(S,V(G)\setminus S)
 &=\boldsymbol{1}_S^{\mathsf T}(dI-A)\boldsymbol{1}_S\\
 &=\boldsymbol{x}^{\mathsf T}(dI-A)\boldsymbol{x}\\
 &\ge(d-\lambda_2(G))\|\boldsymbol{x}\|_2^2\\
 &=(d-\lambda_2(G))\frac{|S|(n-|S|)}{n}\\
 &\ge\frac{\eps d}{2}|S|.
\end{aligned}
\]
Thus $h(G)\ge\frac{\eps d}{2}$. Applying Theorem~\ref{thm:main} with $\frac{\eps}{2}$ for $\eps$ proves the corollary.
\end{proof}

Section~\ref{sec:preliminaries} records the probabilistic and extremal tools used in the proofs.  In Section~\ref{sec:construction}, we construct many pairwise disjoint small connected sets such that, after the deletion of
any sufficiently small vertex set from their external neighbourhoods, all but a small proportion of these neighbourhoods remain large.  In Section~\ref{sec:main-proof}, we combine these connected sets with connected hitting sets to construct minors and prove Theorem~\ref{thm:main}, where a  \emph{connected set} is 
a vertex subset  that induces a connected subgraph and a  \emph{hitting set} is a set that intersects every member of a given family of sets. 

\section{Preliminaries}\label{sec:preliminaries}

For $r\in\N$, write $[r]=\{1,\ldots,r\}$. For
$X\subseteq V(G)$, let $G[X]$ be the subgraph induced by $X$, and
let $N_G(X)$ be the external neighbourhood of $X$, that is the set of vertices in $V(G)\setminus X$ with a
neighbour in $X$. 
%A vertex set is called connected if it induces a
%connected subgraph. 
A set $T$ \emph{meets} a set $W$ if
$T\cap W\ne\emptyset$. We write $\delta(G)$ and $\Delta(G)$ for the
minimum and maximum degrees, and $\deg_G(v)$ for the degree of a
vertex $v\in V(G)$. The average degree of a nontrivial graph
$G$ is $\overline d(G)=\frac{2|E(G)|}{|V(G)|}$. The symbol
$\mathbin{\dot\cup}$ denotes a union of pairwise disjoint sets.

\begin{lemma}
[Krivelevich--Nenadov~{\cite[Lemma~3.1]{KN}}]
\label{lem:connected-hitting}
For every $\alpha\in(0,\frac{1}{2})$, there exist a constant
$K=O(\alpha^{-3})$ and an integer $N_0$ such that the following
holds. Let $H$ be a graph on $N\ge N_0$ vertices with maximum degree
at most $d$ and $h(H)\ge\alpha d$. Let $s>0$ and let $q\le N$ be a
positive integer such that $sq\ge2N$. If
$W_1,\ldots,W_q\subseteq V(H)$ and $|W_i|\ge s$ for every
$i\in[q]$, then there is a connected set $T\subseteq V(H)$ which
meets every $W_i$ (i.e., $T$ is a connected hitting set) and satisfies
\[
 |T|\le K\cdot \frac{N}{s}\log\frac{qs}{N}.
\]
\end{lemma}

A \emph{lazy random walk of length $t$} in a graph $H$ is a Markov
chain $(X_0,\ldots,X_t)$ with the following transition rule. 
For
$0\le j<t$, conditional on $X_j=v$, if $\deg_H(v)>0$, $X_{j+1}=v$ with probability $\frac{1}{2}$ and $X_{j+1}=u\in N_H(v)$ with probability
$\frac{1}{2\deg_H(v)}$. If $v$ is isolated, $X_{j+1}=v$ with
probability $1$. 
When $E(H)\ne\emptyset$, the distribution
\[
  \pi_H(v)=\frac{\deg_H(v)}{2|E(H)|} \text{ for } v\in V(H)
\]
is stationary: if $X_0$ has distribution $\pi_H$, then $X_j$ has
distribution $\pi_H$ for every $0\le j\le t$. We call the walk
\emph{stationary} when $X_0$ has distribution $\pi_H$. 

\begin{lemma}
[Krivelevich--Nenadov~{\cite[Lemma~2.4]{KN}}]
\label{lem:avoid}
Let \(H\) be a graph on \(N\ge 2\) vertices, and let
\(\Delta=\Delta(H)\ge 1\).
For every \(W\subseteq V(H)\) and every positive integer \(t\),
the probability that a stationary lazy random walk
\((X_0,\ldots,X_t)\) in \(H\) avoids \(W\), that is,
$X_j\notin W$ for every $0\le j\le t,$
is at most
\[
    \exp\left(
        -\frac{h(H)^3|W|t}{8\Delta^3N}
    \right).
\]
\end{lemma}

\begin{corollary}\label{cor:partial-hitting}
Let $H$ be a graph on $N\ge2$ vertices with maximum degree at most
$d$ and $h(H)\ge\alpha d$, where $d\ge1$ and $\alpha>0$. Let
$q\in\N$ and
$W_1,\ldots,W_q\subseteq V(H)$ satisfy $|W_i|\ge w\ge1$ for every
$i\in[q]$. For every nonnegative integer $t$, there is a connected
set $T\subseteq V(H)$ with $|T|\le t+1$ which meets at least
\[
 q\left(1-\exp\left\{-\frac{\alpha^3wt}{8N}\right\}\right)
\]
of the sets $W_1,\ldots,W_q$.
\end{corollary}

\begin{proof}
The case \(t=0\) is immediate, so we may assume that \(t\ge1\). Let \((X_0,\ldots,X_t)\) be a stationary lazy
random walk in \(H\).

Let \(\Delta=\Delta(H)\). Since
$h(H)\ge \alpha d>0,$
the graph \(H\) is connected and \(\Delta\ge1\). Moreover,
\(\Delta\le d\), and hence
$\frac{h(H)}{\Delta}\ge \frac{\alpha d}{\Delta}\ge \alpha.$

For each \(i\in[q]\), let \(I_i\) be the indicator of the event
that the walk avoids \(W_i\). By Lemma~\ref{lem:avoid},
\[
\begin{aligned}
\mathbb E[I_i]
&=\Pr\bigl(X_j\notin W_i
          \text{ for every }0\le j\le t\bigr)\\
&\le
\exp\left(
    -\frac{h(H)^3|W_i|t}{8\Delta^3N}
\right)\\
&\le
\exp\left(
    -\frac{\alpha^3|W_i|t}{8N}
\right)\\
&\le
\exp\left(
    -\frac{\alpha^3wt}{8N}
\right).
\end{aligned}
\]
Therefore, if
$Y=\sum_{i=1}^q I_i$
is the number of sets \(W_i\) avoided by the walk, then linearity
of expectation gives
\[
    \mathbb E[Y]
    \le
    q\exp\left(
        -\frac{\alpha^3wt}{8N}
    \right).
\]
Hence there is a realization
\((x_0,\ldots,x_t)\) of the walk which avoids at most this many
sets. For this realization, let
$T=\{x_0,\ldots,x_t\}.$
Then \(|T|\le t+1\). Moreover, \(H[T]\) is connected, since for
every \(0\le j<t\), either \(x_{j+1}=x_j\) or
\(x_jx_{j+1}\in E(H)\). Finally, \(T\) meets precisely those sets
visited by the walk, and hence it meets at least
\[
    q\left(
        1-\exp\left(
            -\frac{\alpha^3wt}{8N}
        \right)
    \right)
\]
of the sets \(W_1,\ldots,W_q\).
\end{proof}

We will also use the following theorem to obtain a complete minor from a lower bound on average degree.

\begin{theorem}
[Kostochka--Thomason~{\cite{Kostochka,Thomason1984}}]
\label{thm:KT}
There is an absolute constant $c_{\mathrm{KT}}>0$ such that every graph of
average degree at least $r\ge3$ contains a complete minor of order at
least $\frac{c_{\mathrm{KT}}r}{\sqrt{\log r}}$.
\end{theorem}

Applying the definition of $h(G)$ to singleton sets gives
$\delta(G)\ge h(G)$. We use this observation without further mention.

\section{Connected sets with large neighbourhoods}\label{sec:construction}

In this section, we establish a series properties  under the conditions of Theorem \ref{thm:main}.

We first show that, after excluding a small vertex set, we can retain a large induced subgraph with a lower bound on its Cheeger constant. The last part allows us to choose nested induced subgraphs when more vertices are excluded.

\begin{lemma}\label{lem:remainder}
Let $0<\eps\le1$, let $d>0$, and let $G$ be a graph on $n$ vertices
satisfying $\Delta(G)\le d$ and $h(G)\ge\eps d$. Suppose that
\[
 V(G)=Y\mathbin{\dot\cup}D_0\mathbin{\dot\cup}U_0,
 \ |Y|\le\frac{\eps^2n}{16},
 \ |D_0|\le\frac{n}{2}
 \  \text{ and }\ e_G(D_0,U_0)\le\frac{\eps d}{4}|D_0|.
\]
Then there is a partition
$V(G)=Y\mathbin{\dot\cup}D\mathbin{\dot\cup}U$ such that
$D_0\subseteq D$ and
\[
 e_G(D,U)\le\frac{\eps d}{4}|D|,
 \ |D|\le\frac{4|Y|}{3\eps},
 \ |U|\ge\frac{n}{2}
 \ \text{ and } \  h(G[U])\ge\frac{\eps d}{4}.
\]
Moreover, if $Z\subseteq U$ and
$|Y\cup Z|\le\frac{\eps^2n}{16}$, then there is a partition
$
 V(G)=(Y\cup Z)\mathbin{\dot\cup}D'\mathbin{\dot\cup}U'
$
such that $D\subseteq D'$, $U'\subseteq U\setminus Z$, and
\[
 e_G(D',U')\le\frac{\eps d}{4}|D'|,
 \ |D'|\le\frac{4|Y\cup Z|}{3\eps},
 \ |U'|\ge\frac{n}{2}
 \ \text{ and }\  h(G[U'])\ge\frac{\eps d}{4}.
\]
\end{lemma}

\begin{proof}
Choose $D$ inclusion-maximal among the sets
$R\subseteq V(G)\setminus Y$ which contain $D_0$, have at most
$\frac{n}{2}$ vertices, and satisfy
the inequality
$e_G(R,V(G)\setminus(Y\cup R))\le\frac{\eps d}{4}|R|$.
Such a set exists because $D_0$ has these properties. Let
$U=V(G)\setminus(Y\cup D)$.

If $D$ is nonempty, then $h_G(D)\ge h(G)\ge  \eps d$, so
\[
 \eps d|D|
 \le e_G(D,U)+e_G(D,Y)
 \le\frac{\eps d}{4}|D|+d|Y|.
\]
Thus  $|D|\le\frac{4|Y|}{3\eps}$ whether
$D$ is empty or not. Since $|Y|\le\frac{\eps^2n}{16}$ and
$\eps\le1$, we obtain
\[
 |U|\ge n-\left(1+\frac{4}{3\eps}\right)|Y|
 \ge n-\left(\frac{\eps^2}{16}+\frac{\eps}{12}\right)n
 \ge\frac{n}{2}.
\]

Suppose that $h(G[U])<\frac{\eps d}{4}$. The definition of the Cheeger
constant gives a nonempty set $S\subseteq U$ such that
$|S|\le\frac{|U|}{2}$ and
$e_G(S,U\setminus S)<\frac{\eps d}{4}|S|$. Since
$|S|\le\frac{n}{2}$, $h(G)\ge \eps  d$ and $\Delta(G)\le d$, we have
\[
 \eps d|S|
 \le e_G(S, V(G)\setminus S)=e_G(S,U\setminus S)+e_G(S,D)+e_G(S,Y)
 <\frac{\eps d}{4}|S|+d|D|+d|Y|,
\]
so 
\[
|S|<\frac{4(|D|+|Y|)}{3\eps}\le \frac{4|Y|}{3\eps}+\frac{16|Y|}{9\eps^2}.
\]
It then follows that
\[
 |D\cup S|
 <\frac{8|Y|}{3\eps}+\frac{16|Y|}{9\eps^2}
 \le\left(\frac{\eps}{6}+\frac{1}{9}\right)n
 <\frac{n}{2}.
\]
Moreover, we have
\[
 \begin{aligned}
 e_G(D\cup S,U\setminus S)
 &=e_G(D,U\setminus S)+e_G(S,U\setminus S)\\
 &\le e_G(D,U)+e_G(S,U\setminus S)<\frac{\eps d}{4}|D|+\frac{\eps d}{4}|S|=
\frac{\eps d}{4}|D\cup S|.
 \end{aligned}
\]
Thus $D\cup S$ satisfies all the conditions used to choose $D$ and
strictly contains $D$, a contradiction. Hence
$h(G[U])\ge\frac{\eps d}{4}$. This proves the first part. 

Next, we show  the second part.  Note that
\[
|D|\le \frac{4}{3\eps}|Y|\le \frac{4}{3\eps}\cdot \frac{\eps^2 n}{16}=\frac{\eps n}{12}\le \frac{n}{2}.
\]
Let $Y_1=Y\cup Z$, $D_1=D$ and $U_1=U\setminus Z$.  Then
\[
 V(G)=Y_1\mathbin{\dot\cup}D_1\mathbin{\dot\cup}U_1,
 \quad |Y_1|\le\frac{\eps^2n}{16},
 \quad |D_1|\le\frac{n}{2},
 \quad e_G(D_1,U_1)\le e_G(D,U)\le \frac{\eps d}{4}|D_1|.
\]
That is,  $Y_1$, $D_1$ and $U_1$ satisfy the  initial hypotheses
 on $Y_0:=Y$, $D_0$ and $U_0$.
Repeating the above inclusion-maximal choice with $Y_1$, $D_1$, and
$U_1$ in place of $Y_0$, $D_0$, and $U_0$, respectively, 
gives sets $D'$ and $U'$ with $D_1\subseteq D'$ and
$U'\subseteq U_1$. Then the second part follows from the above argument.
\end{proof}

The next lemma finds a small  connected set for which the size of the external neighbourhood has a lower bound that is linear in both its size and the maximum degree.

\begin{lemma}\label{lem:large-neighbourhood}
Let $0<\sigma\le1$, let $d\ge4$ satisfy $\sigma d\ge2$, and let $G$
be a graph on $N$ vertices satisfying $\Delta(G)\le d$ and
$h(G)\ge\sigma d$. For every integer $\ell$ with
$1\le\ell\le\frac{N}{d}$, there is a nonempty connected set
$A\subseteq V(G)$ such that $|A|\le\ell$ and
$|N_G(A)|\ge\frac{\sigma d\ell}{2}$.
\end{lemma}

\begin{proof}
Since $h(G)>0$,  $G$ is connected. Repeatedly delete a leaf
from a spanning tree of $G$ until $\ell$ vertices remain, and let $R$
be the remaining vertex set. Evidently, $G[R]$ is connected, so 
$|E(G[R])|\ge \ell-1$.
Since  $\ell\le\frac{N}{d}<\frac{N}{2}$ and $h(G)\ge\sigma d$, we obtain
\[
 \sigma d\ell
 \le e_G(R,V(G)\setminus R)=\sum_{v\in R}\deg_G(v)-2|E(G[R])|
 \le d\ell-2(\ell-1)\le N-2(\ell-1), 
\]
so
$\frac{\sigma d\ell}{2}\le\frac{N}{2}-\ell+1$.

We construct $A$ iteratively. Start with any vertex $v$ of $G$ and let $A_1=\{v\}$. Then  $|N_G(A_1)|=\deg_G(v)=h_G(A_1)\ge h(G)\ge\sigma d \ge\frac{\sigma d}{2}$. For $1\le i\le \ell-1$, let $A_{i+1}=A_i\cup \{x_i\}$ with $x_i\in N_G(A_i)$.  
Suppose that, for some $i$,
$|A_i|<\ell$ and $|N_G(A_i)|<\frac{\sigma d\ell}{2}$.  Let
$B_i=N_G(A_i)$ and $S_i=A_i\cup B_i$.  If $|S_i|>\frac{N}{2}$, then
as $|A_i|\le\ell-1$, 
\[
 |N_G(A_i)|=|S_i|-|A_i|>\frac{N}{2}-\ell+1
 \ge\frac{\sigma d\ell}{2},
\]
a contradiction. Thus $|S_i|\le \frac{N}{2}$. 
Since $h(G)\ge \sigma d$ and $|S_i|\le \frac{N}{2}$, we have 
$e_G(S_i,V(G)\setminus S_i)\ge \sigma d |S_i|$.
No vertex of $A_i$ has a neighbour outside $S_i$. So every edge from $S_i$
to $V(G)\setminus S_i$  has its endpoint in $S_i$ belonging to
$B_i$. So
\[
 \sum_{x\in B_i}|N_G(x)\setminus S_i|
 =e_G(S_i,V(G)\setminus S_i)
 \ge\sigma d|S_i|\ge\sigma d|B_i|.
\]
Hence for some $x_i\in B_i$,  $|N_G(x_i)\setminus S_i|\ge \sigma d$. 
Let  $A_{i+1}=A_i\cup\{x_i\}$. Then $A_{i+1}$  is connected, and 
$N_G(A_{i+1})$ 
contains the disjoint sets $B_i\setminus\{x_i\}$ and
$N_G(x_i)\setminus S_i$. So
$|N_G(A_{i+1})|\ge |B_i|-1+|N_G(x_i)\setminus S_i|\ge |B_i|-1+\sigma d$, i.e., $|N_G(A_{i+1})|-|N_G(A_i)|\ge \sigma d-1\ge \frac{\sigma d}{2}$. 

If at any step the bound $|N_G(A_i)|\ge \frac{\sigma d\ell}{2}$ is reached, we are done by setting $A=A_i$. Otherwise, after $\ell-1$ expansions (so $|A_\ell|=\ell$), for $A=A_{\ell}$, we have 
\[
|N_G(A)|
\ge \sigma d + (\ell-1)\frac{\sigma d}{2}
= \frac{\sigma d(\ell+1)}{2}
\ge \frac{\sigma d\ell}{2}. \qedhere
\]
\end{proof}

The next lemma  combines the preceding two lemmas to construct pairwise disjoint connected sets with large external neighbourhoods.
After the deletion of any sufficiently small vertex set, all but a small proportion of these neighbourhoods remain large.

\begin{lemma}\label{lem:stable-branches}
For every $\eps\in(0,1]$, there are constants
$c=c(\eps)>0$, $\xi=\xi(\eps)>0$, $\zeta=\zeta(\eps)>0$, and an
integer $d_2=d_2(\eps)$ such that 
\[
c\le\frac{1}{32}, \
\xi\le\frac{1}{2048}\ \text{ and } \ \zeta\le\eps,
\]
and
the following holds. Let $G$ be a
graph on $n$ vertices satisfying $\Delta(G)\le d$,
$h(G)\ge\eps d$, and $d\ge d_2$. If $\ell$ is an integer satisfying
$2\le\ell\le\frac{n}{2d}$, then there exist an integer $m\ge1$, pairwise
disjoint nonempty connected sets $A_1,\ldots,A_m$, and sets
$Q_1,\ldots,Q_m$ with the following properties. If
$X=\bigcup_{i=1}^mA_i$, then
\begin{enumerate}[label=\textup{(\roman*)}]
 \item $\frac{\xi n}{\ell}\le m\le\frac{4\xi n}{\ell}$ and
       $|X|\le\frac{\eps\zeta n}{32}$;
 \item $|A_i|\le\ell$, $Q_i\subseteq N_G(A_i)\setminus X$, and
       $|Q_i|\ge cd\ell$ for every $i\in[m]$;
 \item for every $Z\subseteq V(G)$ with $|Z|\le\zeta n$, fewer than
       $\frac{m}{16}$ indices $i\in[m]$ satisfy
       $|Q_i\setminus Z|<cd\ell$.
\end{enumerate}

\end{lemma}

\begin{proof}
Let
\[
 \zeta=\frac{\eps^3}{2^{22}}, \ 
 p=\min\left\{\frac{1}{4},\frac{\eps}{6400},
                 \frac{4\zeta}{\eps}\right\}, \
 c=\frac{\eps}{32} \ \text{ and } \
 \xi=\frac{p\eps^2}{512}.
\]
Since $0<\eps\le1$ and $p\le\frac{1}{4}$, these constants satisfy
$c\le\frac{1}{32}$, $\xi\le\frac{1}{2048}$ and $\zeta\le\eps$.
Choose
$d_2\ge\left\lceil\frac{32}{\eps^2}\right\rceil$ so large that
\[
 \frac{\eps^2d_2}{16}\ge\frac{24}{p}+1,
 \ \frac{\eps d_2}{4}\ge2 
 \ \text{ and } \ pd_2\ge12.
\]
Let 
$x=\frac{\eps^2n}{32\ell}$ and
$y=\frac{\eps d\ell}{8}$.
Set $M=\lfloor x\rfloor$ and $s=\lfloor y\rfloor$.
Since $\frac{n}{\ell}\ge2d$ and
$d\ge d_2$, we have
\[
 x\ge\frac{\eps^2d}{16}\ge\frac{24}{p}+1
 \ \text{ and } \ 
 y\ge\frac{\eps d}{4}\ge2.
\]
As $\lfloor z\rfloor\ge z/2$ for $z\ge2$ and  $\lfloor a\rfloor-a\ge -1$ for any real $a$, we have
$M\ge\frac{\eps^2n}{64\ell}$, 
$M\ge\frac{24}{p}$, 
$s\ge\frac{\eps d\ell}{16}$  and
$pd\ge 12$.

Set $\sigma=\frac{\eps}{4}$. Since
$d\ge d_2\ge\frac{32}{\eps^2}$ and $0<\eps\le1$, we have
$d\ge4$ and $\sigma d\ge 2$.
We first construct recursively pairwise disjoint nonempty connected sets
$A_1,\ldots,A_M$ and sets $W_1,\ldots,W_M$ such that
$|A_i|\le\ell$, $W_i\subseteq N_G(A_i)$, and $|W_i|=s$
for every $i\in[M]$. 
For $0\le j<M$, 
suppose that
$A_1,\ldots,A_j$ and $W_1,\ldots,W_j$ have been chosen. Let
$Y_j=\bigcup_{i=1}^jA_i$ with $Y_0=\emptyset$. Since
\[
 |Y_j|\le j\ell<M\ell\le\frac{\eps^2n}{32},
\]
we apply Lemma~\ref{lem:remainder} with
$Y=Y_j$, $D_0=\emptyset$ and $U_0=V(G)\setminus Y_j$ to yield  a set
$U_j\subseteq V(G)\setminus Y_j$ satisfying
\[
 |U_j|\ge\frac{n}{2}
 \ \text{ and } \ 
 h(G[U_j])\ge\frac{\eps d}{4}.
\]
Since
$\ell\le\frac{n}{2d}\le\frac{|U_j|}{d}$, we apply
Lemma~\ref{lem:large-neighbourhood} to $G[U_j]$ with
$\sigma=\frac{\eps}{4}$ to yield  a nonempty connected set
$A_{j+1}\subseteq U_j$ such that
\[
 |A_{j+1}|\le\ell
 \ \text{ and } \
 |N_{G[U_j]}(A_{j+1})|\ge\frac{\eps d\ell}{8}.
\]
Choose a set
$W_{j+1}\subseteq N_{G[U_j]}(A_{j+1})$ with $|W_{j+1}|=s$.
Since $A_{j+1}\subseteq U_j\subseteq V(G)\setminus Y_j$, the set
$A_{j+1}$ is disjoint from $A_1,\ldots,A_j$. Thus the same
construction can be performed for every $j=0,\ldots,M-1$, producing
pairwise disjoint connected sets $A_1,\ldots,A_M$ and the 
sets $W_1,\ldots,W_M$, as required.

Fix $v\in V(G)$. For every index $i$ with $v\in W_i$,
choose an edge joining $v$ to $A_i$. The endpoints of these edges in
the sets $A_i$ are distinct, because the sets $A_1,\ldots,A_M$ are
pairwise disjoint. Hence the number of indices $i$ for which
$v\in W_i$ is at most $\deg_G(v)\le d$. That is, every vertex $v \in V(G)$ belongs to at most $d$ of the sets $W_1,\dots,W_M$.

%We next trim a subfamily of the sets $W_i$ so that no vertex occurs too many times. 
%We next select some of the sets $W_i$ and remove vertices from them so that each vertex belongs to at most $\lfloor3pd\rfloor$ of the resulting sets. 

Let $\Lambda=\lfloor3pd\rfloor$ and
$m_0=\left\lfloor\frac{pM}{3}\right\rfloor$.
We claim that, for every integer $j$ with $0\le j\le m_0$,
there are an index set $\mathcal I_j\subseteq[M]$ with
$|\mathcal I_j|=j$ and sets $R_i\subseteq W_i$ for
$i\in\mathcal I_j$ such that
\[
 |R_i|>\frac{7s}{8}
 \ \text{ for every }i\in\mathcal I_j
\]
and
\[
 \bigl|\{i\in\mathcal I_j:v\in R_i\}\bigr|\le\Lambda
\ \text{ for every }v\in V(G).
\]
We prove the claim by induction on $j$.
For $j=0$, $\mathcal I_0=\emptyset$, so the claim holds trivially. Suppose that  the claim holds 
for $0\le j<m_0$. That is, for every $j$ with  $0\le j<m_0$, there
is an index set $\mathcal I_j\subseteq[M]$ with 
$|\mathcal I_j|=j$ and sets $R_i\subseteq W_i$ for
$i\in\mathcal I_j$ such that 
 $|R_i|>\frac{7s}{8}$ 
for every $i\in\mathcal I_j$ and 
$\lambda_j(v):=\bigl|\{i\in\mathcal I_j:v\in R_i\}\bigr|\le\Lambda$ for every $v\in V(G)$. 
Let $S_j=\{v\in V(G):\lambda_j(v)=\Lambda\}$.
Since $R_i\subseteq W_i$ and $|W_i|=s$ for every
$i\in\mathcal I_j$, we have
\[
 \Lambda|S_j|
 \le\sum_{v\in V(G)}\lambda_j(v)
 =\sum_{i\in\mathcal I_j}|R_i|
 \le js.
\]
Moreover, since every vertex belongs to at most $d$ of the sets $W_1,\ldots,W_M$, we have
\[
 \frac{1}{M-j}
 \sum_{i\in[M]\setminus\mathcal I_j}|W_i\cap S_j|
 \le\frac{d|S_j|}{M-j}
 \le\frac{djs}{\Lambda(M-j)}.
\]
Since $pd\ge12$, we have
$\Lambda=\lfloor 3pd\rfloor\ge3pd-\frac{pd}{12}=\frac{35pd}{12}$.
Also, since $j<m_0\le\frac{pM}{3}$ and $p\le\frac14$, we have
$M-j>\frac{11M}{12}$. Thus
\[
 \frac{djs}{\Lambda(M-j)}
 <\frac{d(pM/3)s}{(35pd/12)(11M/12)}
 =\frac{48s}{385}<\frac{s}{8}.
\]
It follows that there is an index $i_{j+1}\in[M]\setminus\mathcal I_j$ such that
$|W_{i_{j+1}}\cap S_j|<\frac{s}{8}$.
Let 
\[
 \mathcal I_{j+1}=\mathcal I_j\cup\{i_{j+1}\} \ 
\text{ and } \ 
 R_{i_{j+1}}=W_{i_{j+1}}\setminus S_j.
\]
Then $|\mathcal I_{j+1}|=j+1$ and 
\[
 |R_{i_{j+1}}|=|W_{i_{j+1}}\setminus (W_{i_{j+1}}\cap S_j)|
 =s-|W_{i_{j+1}}\cap S_j|
 >\frac{7s}{8}.
\]
Fix $v\in V(G)$.
If $v\in S_j$, then $v\notin R_{i_{j+1}}$, so
$\bigl|\{i\in\mathcal I_{j+1}:v\in R_i\}\bigr|
 =\lambda_j(v)=\Lambda$.
If $v\notin S_j$, then $\lambda_j(v)\le\Lambda$ by the
induction hypothesis and $\lambda_j(v)\ne\Lambda$ by the
definition of $S_j$. Since $\lambda_j(v)$ is an integer,
$\lambda_j(v)\le\Lambda-1$. The new set $R_{i_{j+1}}$
can increase the number of sets containing $v$ by at most one,
so
$\bigl|\{i\in\mathcal I_{j+1}:v\in R_i\}\bigr|
 \le\lambda_j(v)+1\le\Lambda$. 
In either case,  $\bigl|\{i\in\mathcal I_{j+1}:v\in R_i\}\bigr|\le\Lambda$.
By the induction hypothesis again, there is an index set $\mathcal I_{j+1}\subseteq[M]$ with
$|\mathcal I_{j+1}|=j+1$ and sets $R_i\subseteq W_i$ for
$i\in\mathcal I_{j+1}$ such that
\[
 |R_i|>\frac{7s}{8}
 \ \text{ for every }i\in\mathcal I_{j+1}
\]
and
\[
 \bigl|\{i\in\mathcal I_{j+1}:v\in R_i\}\bigr|\le\Lambda
\ \text{ for every }v\in V(G).
\]
This completes the induction.

Apply the claim with $j=m_0$, and set
$\mathcal I=\mathcal I_{m_0}$.
Then $|\mathcal I|=m_0$ and $R_i\subseteq W_i$ for 
$i\in\mathcal I$ such that 
\begin{equation}\label{a1}
 |R_i|>\frac{7s}{8}
 \ \text{ for every }i\in\mathcal I
\end{equation}
and
\begin{equation}\label{a2}
 \bigl|\{i\in\mathcal I:v\in R_i\}\bigr|\le\Lambda
\ \text{ for every }v\in V(G).
\end{equation}
Let 
\[
 X_0=\bigcup_{i\in\mathcal I}A_i \
  \text{ and } \
 \mathcal B=\left\{i\in\mathcal I:
             |R_i\cap X_0|>\frac{s}{8}\right\}.
\]
Using \eqref{a2} and
$|X_0|\le m_0\ell$, we obtain
\[
 \begin{aligned}
 \frac{|\mathcal B|s}{8}
 &\le\sum_{i\in\mathcal I}|R_i\cap X_0|
   =\sum_{v\in X_0}\bigl|\{i\in\mathcal I:v\in R_i\}\bigr|\\
 &\le\Lambda|X_0|
 \le3pd\,m_0\ell
 \le p^2dM\ell,
 \end{aligned}
\]
which, together with $s\ge\frac{\eps d\ell}{16}$ and
$p\le\frac{\eps}{6400}$,  gives
$|\mathcal B|\le\frac{pM}{50}$. 
Furthermore, since $pM\ge24$, we have
\[
 m_0\ge\frac{pM}{3}-1\ge\frac{7pM}{24}.
\]
It follows that
\[
 |\mathcal I\setminus\mathcal B|
 \ge\frac{7pM}{24}-\frac{pM}{50}
 =\frac{163pM}{600}>\frac{pM}{4}.
\]
So we can 
choose  a set
$\mathcal I'=\{i_1,\ldots,i_m\}\subseteq
\mathcal I\setminus\mathcal B$, where
$m=\left\lfloor\frac{pM}{4}\right\rfloor$. For each
$j\in[m]$, relabel the triple
$(A_{i_j},W_{i_j},R_{i_j})$ as $(A_j,W_j,R_j)$. Finally, let
\[
 X=\bigcup_{j=1}^mA_j 
\  \text{ and } \  
 Q_j=R_j\setminus X_0 \ \text{ for } j\in[m].
\]

Since $pM\ge24$, we have
\[
 m=\left\lfloor\frac{pM}{4}\right\rfloor \ge\frac{pM}{8}
 \ge\frac{p\eps^2n}{512\ell}
 =\frac{\xi n}{\ell}.
\]
On the other hand, we have
\[
 m=\left\lfloor\frac{pM}{4}\right\rfloor \le\frac{pM}{4}
 \le\frac{p\eps^2n}{128\ell}
 =\frac{4\xi n}{\ell},
\]
so with $p\le\frac{4\zeta}{\eps}$, we have
\[
 |X|\le m\ell\le\frac{p\eps^2n}{128}
 \le\frac{\eps\zeta n}{32}.
\]
This proves property (i).

For every
$j\in[m]$, $Q_j\subseteq W_j\setminus X_0
       \subseteq N_G(A_j)\setminus X$, and 
since $X\subseteq X_0$ and $i_j\notin\mathcal B$,  from \eqref{a1} and the definition of $\mathcal B$, we have
\[ 
 |Q_j|=|R_j\setminus (R_j\cap X_0)|>\frac{7s}{8}-\frac{s}{8}
       =\frac{3s}{4}\ge cd\ell.
\]
This proves property (ii).

It remains to prove property (iii).  Fix a set
$Z\subseteq V(G)$ with $|Z|\le\zeta n$, and define
\[
 \mathcal B_Z=
 \{i\in[m]:|Q_i\setminus Z|<cd\ell\}.
\]
As $s\ge \frac{\eps d\ell}{16}$,
$cd\ell=\frac{\eps d\ell}{32}\le\frac{s}{2}$. Since
$|Q_i|>\frac{3s}{4}$, every $i\in\mathcal B_Z$ satisfies
$|Q_i\cap Z|>\frac{s}{4}$. 
Also, every vertex belongs to at most
$\Lambda\le3pd$ of the sets $Q_1,\ldots,Q_m$. 
Counting these incidences by indices and
then by vertices gives
\[
 \frac{s}{4}|\mathcal B_Z|
 \le\sum_{i\in\mathcal B_Z}|Q_i\cap Z|
 \le3pd|Z|.
\]
As $|Z|\le\zeta n$, $s\ge\frac{\eps d\ell}{16}$  and $m\ge \frac{\xi n}{\ell}=\frac{p\eps^2n}{512\ell}$, we obtain
\[
 \frac{|\mathcal B_Z|}{m}
 \le
 \frac{12pd\zeta n}
      {(\eps d\ell/16)(p\eps^2n/(512\ell))}
 =\frac{98304\zeta}{\eps^3}
 =\frac{3}{128}<\frac{1}{16}.
\]
This proves property (iii).
\end{proof}

The following lemma provides small connected sets meeting
all, or at least three quarters of a family of large vertex sets.

\begin{lemma}\label{lem:scale-hitting}
For every $\alpha\in(0,\frac{1}{2})$, $c\in(0, \frac{1}{32}]$ and
$\xi\in(0,\frac{1}{2048}]$, there is a constant $C=C(\alpha,c)>0$ and an
integer $d_3=d_3(\alpha,c,\xi)\ge3$ such that the following holds.
Let $n,d,\ell$ be positive integers satisfying $d\ge d_3$ and
$2\le\ell\le \frac{n}{2d}$, and let $H$ be a graph on $N$ vertices with
$\frac{n}{2}\le N\le n$, $\Delta(H)\le d$ and $h(H)\ge\alpha d$.
Suppose that $q$ is a positive integer satisfying
$\frac{15\xi n}{16\ell}\le q\le \frac{4\xi n}{\ell}$, and that
$W_1,\ldots,W_q\subseteq V(H)$ satisfy $|W_i|\ge cd\ell$ for
every $i\in[q]$.
\begin{enumerate}[label=\textup{(\alph*)}]
 \item There is a connected set $T\subseteq V(H)$ which
       meets every $W_i$ and satisfies
       $|T|\le \frac{Cn\log d}{d\ell}$.
 \item There is a connected set $T\subseteq V(H)$ which
       meets at least $\frac{3}{4}q$ of the sets $W_1,\ldots,W_q$ and
       satisfies $|T|\le \frac{Cn}{d\ell}$.
\end{enumerate}
%The sets in \textup{(a)} and \textup{(b)} may be different.
\end{lemma}

\begin{proof}
%Fix $\alpha,c,\xi$ as in the statement. 
Let $K=K(\alpha)>0$ and
$N_0=N_0(\alpha)\in\N$ be constants for which
Lemma~\ref{lem:connected-hitting} holds with this value of $\alpha$.
%These constants are uniform over all graphs and families of sets satisfying the hypotheses of that lemma. 
Choose $C$ so that
$C\ge \frac{K}{c}$ and $C\ge \frac{8\log4}{\alpha^3c}+1$, and choose an integer
$d_3\ge\max\{3,N_0\}$ so that $\frac{15\xi cd_3}{16}\ge2$. Then $d\ge d_3\ge N_0$. 
Note that $C$ depends only on $\alpha,c$, and $d_3$ depends only on
$\alpha,c,\xi$.

%Now let $n,d,\ell,H,N,q,W_1,\ldots,W_q$ satisfy the hypotheses.
Since $N\ge \frac{n}{2}$, $\ell\ge2$ and $\ell\le \frac{n}{2d}$, we have
$N\ge \frac{n}{2}\ge d\ell\ge 2d\ge N_0$.

Let $s=cd\ell$. Evidently, $s>0$.

Since $\xi\le1/2048$ and $\ell\ge2$, we have $q\le \frac{4\xi n}{\ell}\le \frac{n}{1024}<\frac{n}{2}\le N$. Also, 
%The lower bound on $q$ and the choice of $d_3$ give
%\[
$qs\ge\frac{15\xi n}{16\ell}\cdot cd\ell
     =\frac{15\xi cd}{16}\cdot n
     \ge2n\ge2N$.

Recall that  $\Delta(H)\le d$,  $h(H)\ge\alpha d$, and 
$W_1,\ldots,W_q\subseteq V(H)$ satisfy $|W_i|\ge s$ for
every $i\in[q]$.

Thus hypotheses of Lemma~\ref{lem:connected-hitting} are satisfied.

For part \textup{(a)}, we have by   Lemma~\ref{lem:connected-hitting} that there is  a connected set $T$
meeting every $W_i$ and 
\[
|T|\le K\cdot \frac{N}{s}\log \frac{qs}{N}.
\] 
Using $N\ge n/2$ and $8\xi c\le1$, we obtain
\[
 2\le\frac{qs}{N}
   \le\frac{2qcd\ell}{n}
   \le8\xi cd\le d.
\]
Note that  $N\le n$ and $C\ge \frac{K}{c}$. Then 
\[
 |T|\le\frac{Kn\log d}{cd\ell}
      \le\frac{Cn\log d}{d\ell}.
\]
This proves part (a).
\
For part (b), let
$t=\left\lceil \frac{8N\log4}{\alpha^3cd\ell}\right\rceil$.
Since $qs\ge 2N$ and $q\le N$, we have $s\ge 2$. 
Note that  $N\ge2$ and  $d\ge3$. By 
Corollary~\ref{cor:partial-hitting}, there is a connected set $T$ with $|T|\le t+1$ which meets  at least
\[
 q\left(1-\exp\left(-\frac{\alpha^3cd\ell\,t}{8N}\right)\right)
 \ge q\bigl(1-\exp(-\log4)\bigr)
 =\frac{3q}{4}
\]
of the sets $W_1,\dots, W_q$. Since $N\le n$ and $\frac{n}{d\ell}\ge2$, we have
\[
 |T|\le t+1
      \le\frac{8N\log4}{\alpha^3cd\ell}+2
      \le\left(\frac{8\log4}{\alpha^3c}+1\right)
             \frac{n}{d\ell}
      \le\frac{Cn}{d\ell}.
\]
%The set $T$ is nonempty because it meets at least one of the sets.
This proves part \textup{(b)}.
\end{proof}

\section{Proof of the main theorem}
\label{sec:main-proof}

%We now combine the connected sets from
%Lemma~\ref{lem:stable-branches} with the connected hitting sets from
%Lemma~\ref{lem:scale-hitting} to construct minors. 

The next lemma gives
a complete minor when $d\ell^2\ge n\log d$ and a minor of large
average degree under the weaker condition $d\ell^2\ge n$.

\begin{lemma}\label{lem:scale}
For every $\eps\in(0,1]$, there exist a constant $b=b(\eps)>0$ and an
integer $d_0=d_0(\eps)$ such that the following holds. Let $G$ be a
graph on $n$ vertices, let $d\ge d_0$ be an integer, and suppose that
$\Delta(G)\le d$ and $h(G)\ge\eps d$. If $\ell$ is an integer
satisfying $2\le\ell\le\frac{n}{2d}$, then the following statements hold.
\begin{enumerate}[label=\textup{(\roman*)}]
 \item If $d\ell^2\ge n\log d$, then
       $\ccl(G)\ge\frac{bn}{\ell}$.
 \item If $d\ell^2\ge n$, then $G$ has a minor of average degree at
       least $\frac{bn}{\ell}$.
\end{enumerate}
\end{lemma}

\begin{proof}
Let $c,\xi,\zeta$ and $d_2$ be given in
Lemma~\ref{lem:stable-branches}.  In particular, $c\le\frac{1}{32}$,
$\xi\le\frac{1}{2048}$ and $\zeta\le\eps$. Let $\alpha=\eps/4$.
Then let $C,d_3$ be given in  Lemma~\ref{lem:scale-hitting}.

Choose $\eta>0$ such that
\[
 \eta\le\frac{1}{16}
 \ \text{ and }\ 
 4\eta\xi C\le\frac{\eps\zeta}{32}.
\]
Choose  an integer $d_0\ge \max\{d_2, d_3\}$ such that  $\eta\xi d_0\ge1$. 
All these choices depend only on $\eps$. Note that $d\ge d_0$.

By Lemma~\ref{lem:stable-branches}, for some integer $m\ge 1$, there are 
pairwise disjoint nonempty connected sets $A_1,\ldots,A_m$ and
sets $Q_1,\ldots,Q_m$ such that 
\[
 \frac{\xi n}{\ell}\le m\le\frac{4\xi n}{\ell},
 \ |X|\le\frac{\eps\zeta n}{32}
 \ \text{ and } \  Q_i\subseteq N_G(A_i)\setminus X
 \text{ for } i\in[m]
\]
with 
\[
 X=\bigcup_{i=1}^mA_i.
\]
Let $k=\lfloor\eta m\rfloor$. Since $m\ge\xi n/\ell$,
$n/\ell\ge2d$ and $\eta\xi d_0\ge1$, we have
$\eta m\ge2\eta\xi d\ge2$, so
\[
1\le\frac{\eta m}{2}\le k\le\eta m\le\frac m{16}.
\]

For a set $U\subseteq V(G)\setminus X$, call an index $i\in[m]$
\emph{$U$-active} if $|Q_i\cap U|\ge cd\ell$.

\medskip
\noindent\textbf{Claim.}
For each of parts (i) and (ii), under the
corresponding hypothesis, there exist pairwise disjoint connected sets
$T_j$, $j\in[k]$, all disjoint from $X$, and sets
$D_0,\ldots,D_k,U_0,\ldots,U_k$ (these sets
may differ between part (i) and part (ii)) with the following properties. 
For $0\le j\le k$, let
\[
 F_0=\emptyset
 \text{ and }
 F_j=T_1\mathbin{\dot\cup}\cdots\mathbin{\dot\cup}T_j
 \text{ for } j\ge 1.
\]
For every $j$ with $0\le j\le k$,
\begin{equation}\label{eq:scale-recursive-partition}
 V(G)=(X\cup F_j)\mathbin{\dot\cup}D_j
       \mathbin{\dot\cup}U_j,
\end{equation}
\[
 e_G(D_j,U_j)\le\alpha d|D_j|,
 \ 
 |D_j|\le\frac{4|X\cup F_j|}{3\eps},
 \ 
 |U_j|\ge\frac{n}{2}, \
 h(G[U_j])\ge\alpha d,
\]
and at least $\frac{15m}{16}$ indices are $U_j$-active.
For every $j\in [k]$,
\[
 T_j\subseteq U_{j-1},
 \ 
 D_{j-1}\subseteq D_j,
 \ 
 U_j\subseteq U_{j-1}\setminus T_j \
 \text{ and }\ 
 |T_j|\le C\ell.
\]
$T_j$ meets $Q_i\cap U_{j-1}$
for every $U_{j-1}$-active index $i$ in part (i), 
and 
$T_j$ meets $Q_i\cap U_{j-1}$ for at least $\frac34$ of
the $U_{j-1}$-active indices $i$ in part (ii).

\begin{proof}
We construct the desired sets recursively.
For the initial stage, let $F_0=\emptyset$. Recall that
$|X|\le\frac{\eps\zeta n}{32}\le\frac{\eps^2n}{16}$.
So we can apply 
Lemma~\ref{lem:remainder} to the partition
\[
 V(G)=X\mathbin{\dot\cup}\emptyset
       \mathbin{\dot\cup}(V(G)\setminus X)
\]
to obtain the new partition $V(G)=X\mathbin{\dot\cup}D_0
       \mathbin{\dot\cup}U_0$ such that
\[
 e_G(D_0,U_0)\le\alpha d|D_0|,
 \ 
 |D_0|\le\frac{4|X\cup F_0|}{3\eps}, \ 
 |U_0|\ge\frac{n}{2} \ 
 \text{ and } \
 h(G[U_0])\ge\alpha d.
\]
As $|X|\le\frac{\eps\zeta n}{32}$, we have
\[
 |D_0|\le\frac{4|X|}{3\eps}
 \le\frac{\zeta n}{24}<\zeta n.
\]
By Lemma \ref{lem:stable-branches} (ii), $Q_i\cap X=\emptyset$ for every $i\in[m]$. So, from $V(G)=X\mathbin{\dot\cup}D_0
       \mathbin{\dot\cup}U_0$, we have
\[
 Q_i\setminus D_0=Q_i\cap U_0
\ \text{ for }i\in[m].
\]
Now Lemma~\ref{lem:stable-branches} (iii) applied
with $Z=D_0$  shows that at least $\frac{15m}{16}$ indices are
$U_0$-active. Thus the claim holds for $j=0$.

Now suppose that the claim holds  for $0\le j<k$. Let $q_j$ be the number of $U_j$-active indices. By induction hypothesis and the upper bound on $m$, we have
\[
 \frac{15\xi n}{16\ell}
 \le\frac{15m}{16}
 \le q_j\le m\le\frac{4\xi n}{\ell}.
\]
Furthermore, $\frac{n}{2}\le |U_j|\le n$, $\Delta(G[U_j])\le d$ and
$h(G[U_j])\ge\alpha d$. Let $W_1,\ldots,W_{q_j}$ be the sets $Q_i\cap U_j$
for $i\in[m]$ such that $i$ is $U_j$-active. Then $|W_i|\ge cd\ell$ for each $i\in [q_j]$. 
Note that $\alpha\in(0,1/2)$, $d\ge d_0\ge d_3$ and $2\le\ell\le n/(2d)$. 
So all the hypotheses of Lemma~\ref{lem:scale-hitting} for
$H=G[U_j]$ are satisfied with $q=q_j$.

In part (i), Lemma \ref{lem:scale-hitting} (a) ensures that there is a connected set $T_{j+1}\subseteq U_j$ meeting $Q_i\cap U_j$ for every $U_j$-active index $i$. Since $d\ell^2\ge n\log d$, we have 
$|T_{j+1}|\le \frac {Cn\log d}{d\ell}\le C\ell$; 
In part (ii), Lemma \ref{lem:scale-hitting} (b) ensures that there is a connected set $T_{j+1}$ meeting $Q_i\cap U_j$ for at least $\frac{3}{4}$ of the
$U_j$-active indices. Since $d\ell^2\ge n$, we have 
$|T_{j+1}|\le\frac{Cn}{d\ell}\le C\ell$.
Thus $T_{j+1}$ has the required properties in either part.

Since $T_{j+1}\subseteq U_j$, we have from 
\eqref{eq:scale-recursive-partition} that  the set
$T_{j+1}$ is disjoint from $X\cup F_j\cup D_j$, so 
$T_1,\ldots,T_{j+1}$ are pairwise disjoint and are disjoint from
$X$. Define
\[
 F_{j+1}=F_j\mathbin{\dot\cup}T_{j+1}.
\]
Since every set $T_r$, $1\le r\le j+1$, was supplied by Lemma \ref{lem:scale-hitting}, we have $|T_r|\le C\ell$,
so
\begin{equation}\label{eq:scale-F-bound}
 \begin{aligned}
 |F_{j+1}|
 \le (j+1)C\ell
 \le kC\ell
 \le\eta mC\ell
 \le4\eta\xi Cn
 \le\frac{\eps\zeta n}{32},
 \end{aligned}
\end{equation}
so
\begin{equation}\label{MM}
 |X\cup F_{j+1}|\le |X|+|F_{j+1}|
 \le\frac{\eps\zeta n}{16}
 \le\frac{\eps^2n}{16}.
\end{equation}
Recall that $\zeta\le\eps\le1$. From the induction hypothesis and the fact that 
$X\cup F_j\subseteq X\cup F_{j+1}$, we have
\[
 |D_j|
 \le\frac{4|X\cup F_j|}{3\eps}
 \le\frac{4|X\cup F_{j+1}|}{3\eps}
 \le\frac{\zeta n}{12}<\frac{n}{2}.
\]
Moreover,
\[
 e_G(D_j,U_j\setminus T_{j+1})
 \le e_G(D_j,U_j)
 \le\alpha d|D_j|.
\]
Thus Lemma~\ref{lem:remainder} applies to the partition
\[
 V(G)=(X\cup F_{j+1})\mathbin{\dot\cup}D_j
       \mathbin{\dot\cup}(U_j\setminus T_{j+1})
\]
to form the new partition 
\[
 V(G)=(X\cup F_{j+1})\mathbin{\dot\cup}D_{j+1}
       \mathbin{\dot\cup}U_{j+1}
\]
with 
$ D_j\subseteq D_{j+1}$,  
$U_{j+1}\subseteq U_j\setminus T_{j+1}$,
\[
 e_G(D_{j+1},U_{j+1})\le\alpha d|D_{j+1}|,
 \ 
 |D_{j+1}|\le\frac{4|X\cup F_{j+1}|}{3\eps},
\  
 |U_{j+1}|\ge\frac{n}{2}
 \  \text{ and }\ 
 h(G[U_{j+1}])\ge\alpha d.
\]
%It remains only to verify that sufficiently many indices are
%$U_{j+1}$-active. 
From $|D_{j+1}|\le\frac{4|X\cup F_{j+1}|}{3\eps}$ and \eqref{MM}, we have 
\[
 |D_{j+1}|\le\frac{\zeta n}{12},
\]
which, together with
\eqref{eq:scale-F-bound} gives
 \[
 |D_{j+1}\cup F_{j+1}|
 \le\frac{\zeta n}{12}+\frac{\eps\zeta n}{32}
 <\zeta n.
\]
Since $Q_i\cap X=\emptyset$ and $V(G)=(X\cup F_{j+1})\mathbin{\dot\cup}D_{j+1}
       \mathbin{\dot\cup}U_{j+1}$,
\[
 Q_i\setminus(D_{j+1}\cup F_{j+1})=Q_i\cap U_{j+1}
\ \text{ for every }i\in[m].
\]
Now property (iii) of Lemma~\ref{lem:stable-branches} applied
with $Z=D_{j+1}\cup F_{j+1}$ shows that at least
$\frac{15m}{16}$ indices are $U_{j+1}$-active. This completes the
inductive step and hence the construction through stage $k$.
\end{proof} 
%This proves Claim~2. %\hfill$\square$

Set $b=\frac{45\eta\xi}{128}$, which depends only on $\eps$.
We now prove parts (i) and (ii) with this choice of $b$.

%\medskip
%\noindent\emph{Proof of part (i).}
Suppose first that $d\ell^2\ge n\log d$.  We use the sets constructed
in the  Claim for part (i). At least $\frac{15m}{16}$ indices are
$U_k$-active.  Since  $k\le\frac{m}{16}$, we can choose distinct
$U_k$-active indices $i_1,\ldots,i_k$.

Fix $j,r\in[k]$. Since $U_k\subseteq U_{j-1}$, we have
\[
 |Q_{i_r}\cap U_{j-1}|
 \ge |Q_{i_r}\cap U_k|
 \ge cd\ell.
\]
Thus $i_r$ is $U_{j-1}$-active. By the 
Claim (for part (i)), $T_j$ meets $Q_{i_r}\cap U_{j-1}$.
As $Q_{i_r}\subseteq N_G(A_{i_r})$, there is an edge between
$T_j$ and $A_{i_r}$.

For each $j\in[k]$, let $B_j=A_{i_j}\cup T_j$.
The sets $B_1,\ldots,B_k$ are pairwise disjoint, because the
sets $A_i$ are pairwise disjoint, the sets $T_j$ are pairwise
disjoint, and every $T_j$ is disjoint from $X$.
Since $A_{i_j}$ and $T_j$ are connected
and there is an edge between them, we deduce that each $B_j$ is connected. Moreover, for distinct
$j,r\in[k]$, the edge between $T_j$ and $A_{i_r}$ joins $B_j$
to $B_r$. Hence $(B_j)_{j\in[k]}$ is a $K_k$-minor model in $G$.
Using $m\ge\frac{\xi n}{\ell}$,
we obtain
\[
 \ccl(G)\ge k
 \ge\frac{\eta m}{2}
 \ge\frac{\eta\xi n}{2\ell}
 \ge\frac{bn}{\ell},
\]
where the last inequality follows from $\frac{45}{128}<\frac12$.
This proves part (i).

%\medskip
%\noindent\emph{Proof of part (ii).}
Suppose next that $d\ell^2\ge n$, and use the sets constructed in the
Claim for part (ii). Define a bipartite graph $J$ with parts
$\{A_1,\ldots,A_m\}$ and $\{T_1,\ldots,T_k\}$, joining $A_i$
to $T_j$ precisely when $e_G(A_i,T_j)>0$.
All these vertex sets are pairwise disjoint and connected, so
they form a $J$-minor model in $G$. Thus $J\prec G$.

For each $j\in[k]$, at least $\frac{15m}{16}$ indices are
$U_{j-1}$-active. The meeting property in the Claim for part (ii)
therefore ensures that $T_j$ meets $Q_i\cap U_{j-1}$ for at least
$\frac34\cdot\frac{15m}{16}=\frac{45m}{64}$ distinct indices $i$.
Each such intersection gives an edge between $T_j$ and $A_i$,
since $Q_i\subseteq N_G(A_i)$. Consequently,
\[
 |E(J)|=\sum_{j=1}^k\deg_J(T_j)
 \ge\frac{45km}{64}.
\]
Since $|V(J)|=m+k\le2m$, $k\ge \eta \frac{m}{2}$ and $m\ge \frac{\xi n}{\ell}$, we have
\[
 \overline d(J)
 =\frac{2|E(J)|}{m+k}
 \ge\frac{45k}{64}
 \ge\frac{45\eta m}{128}
 \ge\frac{45\eta\xi n}{128\ell}
 =\frac{bn}{\ell}.
\]
This proves part (ii).
\end{proof}

We are now ready to prove Theorem~\ref{thm:main} by combining Lemma~\ref{lem:scale} with Theorems~\ref{thm:KN-basic} and~\ref{thm:KT}.

\begin{proof}[Proof of Theorem~\ref{thm:main}] It is trivial if 
$\eps>1$. Assume that
$0<\eps\le1$.

Let $b$ and $d_0$ be supplied by Lemma~\ref{lem:scale}. Decrease $b$
if necessary so that $b\le1$, and increase $d_0$ so that
$\eps d_0\ge3$ and $bd_0\ge3$. By
Theorem~\ref{thm:KN-basic}, there are constants
$b_{\mathrm{KN}}=b_{\mathrm{KN}}(\eps)>0$ and
$n_{\mathrm{KN}}=n_{\mathrm{KN}}(\eps)$ such that
$\ccl(G)\ge b_{\mathrm{KN}}\sqrt n$ whenever $n\ge n_{\mathrm{KN}}$.
Choose $n_0=\max\{n_{\mathrm{KN}},d_0\}$.
This explicitly gives all order requirements in the proof.

\noindent
{\bf Case 1.} 
 $3\le d<d_0$. 

Since
$\frac{d}{\log d}\le\frac{d_0}{\log 3}$,
Theorem~\ref{thm:KN-basic} gives
\[
 \ccl(G)\ge b_{\mathrm{KN}}\sqrt{\frac{\log 3}{d_0}}\,
                 \sqrt{\frac{nd}{\log d}}.
\]

\noindent
{\bf Case 2.} $d\ge d_0$. 

Let
\[
 \ell_1=\max\left\{2,
        \left\lceil\sqrt{\frac{n\log d}{d}}\right\rceil\right\}
 \ \text{ and }\ 
 \ell_2=\max\left\{2,
        \left\lceil\sqrt{\frac{n}{d}}\right\rceil\right\}.
\]

\noindent
{\bf Case 2.1.} $\ell_1\le\frac{n}{2d}$. 

Note that
$2\le\ell_1\le\frac{n}{2d}$ and $d\ell_1^2\ge n\log d$.
Furthermore, the inequality $\ell_1\le\frac{n}{2d}$ implies
$\frac{n}{d}\ge4$, so
$\sqrt{\frac{n\log d}{d}}>1$ and
$\ell_1\le2\sqrt{\frac{n\log d}{d}}$. Thus, 
Lemma~\ref{lem:scale} (i) gives
\[
 \ccl(G)\ge\frac{bn}{\ell_1}
 \ge\frac{b}{2}\sqrt{\frac{nd}{\log d}}.
\]

\noindent
{\bf Case 2.2.}  $\ell_1>\frac{n}{2d}$ but
$\ell_2\le\frac{n}{2d}$. 

We again have $\frac{n}{d}\ge4$.
Since
$\left\lceil\sqrt{\frac{n\log d}{d}}\right\rceil
>\frac{n}{2d}$, it follows that
\[
 \sqrt{\frac{n\log d}{d}}>\frac{n}{2d}-1
 \ge\frac{n}{4d},
\]
and hence $\frac{n}{d}<16\log d$. 
Lemma~\ref{lem:scale} (ii) applied with $\ell=\ell_2$
gives a minor $H\prec G$ with average degree at least
\[
 r_0=\frac{bn}{\ell_2}\ge\frac{b}{2}\sqrt{nd}\ge bd\ge3.
\]
Here we used $\ell_2\le2\sqrt{\frac{n}{d}}$ and
$\frac{n}{d}\ge4$. Since $b\le1$ and
$\ell_2\ge\sqrt{\frac{n}{d}}$, we also have
$r_0\le\sqrt{nd}$. The elementary inequality
$16\log d\le d^4$ for $d\ge3$ gives
\[
 \log r_0\le\log\sqrt{nd}
 =\log d+\frac{1}{2}\log \frac{n}{d} <3\log d.
\]
Theorem~\ref{thm:KT}, applied to $H$ with the lower bound $r_0$
yields
\[
 \ccl(G)\ge\frac{c_{\mathrm{KT}}b}{2\sqrt{3}}
             \sqrt{\frac{nd}{\log d}}.
\]

\noindent
{\bf Case 2.3.}  $\ell_2>\frac{n}{2d}$. 

If
$\frac{n}{d}<4$, then trivially $\frac{n}{d}<16$. If
$\frac{n}{d}\ge4$, then
\[
 \sqrt{\frac{n}{d}}>\frac{n}{2d}-1\ge\frac{n}{4d},
\]
and again $\frac{n}{d}<16$. Thus $d>\frac{n}{16}$. The singleton
case of the Cheeger
condition gives average degree at least $\eps d\ge3$. Since
$\log(\eps d)\le\log d$ and $d>\frac{\sqrt{nd}}{4}$,
Theorem~\ref{thm:KT} gives
\[
 \ccl(G)\ge\frac{c_{\mathrm{KT}}\eps}{4}
             \sqrt{\frac{nd}{\log d}}.
\]

Combining all cases above, we  complete the proof by setting
\[
 \beta=\min\left\{
 b_{\mathrm{KN}}\sqrt{\frac{\log 3}{d_0}},\
 \frac{b}{2},\
 \frac{c_{\mathrm{KT}}b}{2\sqrt{3}},\
 \frac{c_{\mathrm{KT}}\eps}{4}
 \right\}. \qedhere
\]
\end{proof}

\bigskip

\noindent {\bf Acknowledgement.}
This work was supported by the National Natural Science Foundation of China (No.~12571364).

\bigskip

\noindent\textbf{Declaration of competing interest}

\noindent There is no competing interest.

\medskip
\noindent\textbf{Data availability}

\noindent There is no data associated with this paper.

\end{document}